\documentclass{ws-ijnt}

\usepackage{amscd, amsmath, amsfonts, amssymb, latexsym, mathrsfs}
\usepackage{graphicx, wsuipa}

\begin{document}

\markboth{Paul L. Butzer and Tibor K. Pog\'any}
{A new approach to Eisenstein series and Hilbert--Eisenstein series}

%
\catchline{}{}{}{}{}
%

\title{A FRESH APPROACH TO CLASSICAL EISENSTEIN SERIES AND THE NEWER HILBERT--EISENSTEIN SERIES}

\author{PAUL L. BUTZER}

\address{{Lehrstuhl A f\"ur Mathematik, RWTH Aachen, D--52056 Aachen, Germany}\\
\email{butzer@rwth-aachen.de}}

\author{TIBOR K. POG\'ANY}

\address{{John von Neumann Faculty of Informatics, \'Obuda University, 1034--Budapest, Hungary\\
Faculty of Maritime Studies, University of Rijeka, HR--51000 Rijeka, Croatia}\\
\email{\rm poganj@pfri.hr}}

\maketitle

\begin{history}
\received{(Day Month Year)}
\accepted{(Day Month Year)}
\end{history}

\begin{center} {\it Dedicated to the memory of Godfrey Harold Hardy, a Discoverer and Mentor of Srinivasa Aiyangar Ramanujan}
\end{center} \medskip

\begin{abstract}
This paper is concerned with new results for the circular Eisenstein series $\varepsilon_r(z)$ as well as with a novel approach to 
Hilbert-Eisenstein series $\mathfrak h_r(z)$, introduced by Michael Hauss in 1995. The latter turn out to be the product of the 
hyperbolic sinh--function with an explicit closed form linear combination of digamma functions. The results, which include 
differentiability properties and integral representations, are established by independent and different argumentations. 
Highlights are new results on the Butzer--Flocke--Hauss Omega function, one basis for the study of Hilbert-Eisenstein series, which have 
been the subject of several recent papers.
\end{abstract}

\keywords{Bernoulli numbers; Conjugate Bernoulli numbers; Butzer--Flocke--Hauss complete Omega function;  
Digamma function; Dirichlet Eta function; Eisenstein series; Exponential generating functions; Hilbert transform; 
Hilbert--Eisenstein series;  Riemann Zeta function}

\ccode{Mathematics Subject Classification 2010: 11B68, 11M36, 33B15, 33E20, 40C10}

\section{Eisenstein series} 

In order to introduce his method for constructing elliptic functions, Ferdinand Gotthold Max Eisenstein (1823--1852) first 
considered the simpler case of trigonometric functions, specifically the series 
   \[ \pi \cot (\pi z) = \dfrac1z + \sum_{k \in \mathbb N} \left( \dfrac1{z+k} + \dfrac1{z-k}\right)\, ,\]
originally discovered by Leonhard Euler in 1748, presented in \cite[\S 178]{Euler} \footnote{It 
is worth mentioning that it is regarded by Konrad Knopp \cite[p. 207]{Knopp} as the "{\it most remarkable expansion in partial 
fractions}". Also, J. Elstrodt \cite{Elstrodt} nominated this partial fraction expansion for the most interesting formula 
involving elementary functions, see also \cite[p. 149]{AignerZieg}.}. Eisenstein introduced the series (later to be 
famously known as {\it Eisenstein series}, see e.g. Weil \cite{Weil0}, \cite{Weil1} and Iwaniec \cite{Iwan})
   \begin{equation} \label{A0}
     \varepsilon_r(z) := \sum_{k \in \mathbb Z} \dfrac1{(z+k)^r}, 
   \end{equation}
which are defined for $z \in \mathbb C \setminus \mathbb Z$ and all $r \in \mathbb N_2 = \{2, 3, \cdots \}$, they 
being a normally convergent series of meromorphic functions in $\mathbb C$. Since these Eisenstein series of order 
$r \in \mathbb N$ do not exist for $r=1$, one defines aesthetically 
   \[ \varepsilon_1(z) =  \underset{k \in \mathbb Z}{\sum\nolimits_e} \, \dfrac1{z+k} 
                      :=  \lim_{N \to \infty} \sum_{|k| \le N} \dfrac1{z+k} 
                       =  \dfrac1z + \underset{k \in \mathbb Z}{\sum\nolimits'} 
                          \left( \dfrac1{z+k}-\dfrac1k\right)\, .\]
One observes that $\varepsilon_1(z) = \pi\, \cot(\pi z)$ (Euler), and by differentiation 
   \begin{equation} \label{FX}
      \varepsilon_2(z) = \dfrac{\pi^2}{\sin^2(\pi z)}, \qquad \varepsilon_3(z) 
                       = \dfrac{\pi^3\, \cot(\pi z)}{\sin^2(\pi z)}\,;
   \end{equation}
this results in the intriguing relation \cite[p. 299]{Remmert}
   \[ \varepsilon_3(z) = \varepsilon_1(z) \cdot \varepsilon_2(z)\, .\]
There immediately arises the question: "do there exist further $r \in \mathbb N_2$ such that $\varepsilon_{r+2}(z) = \varepsilon_{r+1}(z) 
\cdot \varepsilon_r(z)$ is valid?" Our answer is the following result. 

\begin{theorem}
The unique solution in $r \in \mathbb N$ of the equation 
   \begin{equation} \label{FX1}
      \varepsilon_{r+2}(z) = \varepsilon_{r+1}(z) \cdot \varepsilon_r(z), 
                 \qquad \big( z \in \mathbb C \setminus \mathbb Z \big)
   \end{equation}
is $r=1$.
\end{theorem}

\begin{proof}
Obviously $r$ has to be odd. Indeed, setting $z = \tfrac12$ in \eqref{A0}, we can write $\varepsilon_{r+2}\left(\tfrac12\right)$ in terms 
of the Dirichlet's Lambda--function 
   \[ \lambda(r) = \sum_{k \in \mathbb N_0} \frac1{(2k+1)^r}, \qquad \big( r > 1\big)\]
in the form
   \[ \varepsilon_r \left( \tfrac12 \right) = 2^r (1 + (-1)^r)\, \lambda(r)\, .\]
But by this result the initial equation \eqref{FX1} makes sense only for $r$ odd, since for even $r=2\ell, \ell\in \mathbb N$, 
the relation \eqref{FX1} becomes a contradiction 
   \[ 2^{2\ell+3}\, \lambda(2\ell+2) = 0 \cdot 2^{2\ell+1}\, \lambda(2\ell) = 0\,.\]
On the other hand, bearing in mind the essential differentiability property \cite[pp.  6--13]{Weil0}, \cite[p. 299]{Remmert}, 
namely  
   \begin{equation} \label{A20}
      \varepsilon_r(z) = \frac{(-1)^{r-1}}{\Gamma(r)}\, \varepsilon^{(r-1)}_1(z), 
                         \qquad \left( r \in \mathbb N\right)\, ,
   \end{equation}
and accordingly
   \[ \varepsilon''_r(z) = r(r+1)\, \varepsilon_{r+2}(z),\]
we deduce from \eqref{FX1} the nonlinear second order ODE:
   \begin{equation} \label{W3}
      y'' + (r+1)y'\, y = 0, \qquad \big( y = \varepsilon_r(z)\big)\, .
   \end{equation}
Moreover, as the Eisenstein series is 1--periodic in the sense that $\varepsilon_r(z+k) = \varepsilon_r(z)$ for all 
$z \in \mathbb C \setminus \mathbb Z, k \in \mathbb Z$, we are looking for a 1--periodic particular solution of the ordinary 
differential equation \eqref{W3}. It is 
   \[ y = \sqrt{ \frac{2c_1}{r+1}}\, \tanh \left[ \sqrt{ \dfrac{c_1(r+1)}2} 
          \big( z+c_2\big) \right]\, ,\]
where $c_1, c_2$ stand for integration constants. The $\tanh$ function is 
${\rm  i}\pi$--periodic, so 
   \[ \sqrt{ \dfrac{c_1(r+1)}2} =  {\rm i}\,\pi\, ,\]
accordingly
   \begin{equation} \label{W5}
      \varepsilon_r(z) = - \frac{2\pi}{r+1}\,\tan \pi \big(z + c_2\big) \, .
   \end{equation}
Now, we have 
   \begin{align*}
      \varepsilon_r'(z) &= - \frac{2\pi^2}{r+1} \cdot \frac1{\cos^2\pi(z+c_2)}\\ 
         &= -\frac{2\pi^2}{r+1} \, \Big[\tan^2\pi(z + c_2) + 1\Big] \\
         &= -\frac{r+1}2\, \varepsilon_r^2(z) - \frac{2\pi^2}{r+1} \, ,
   \end{align*}
which coincides exactly for $r=1$ with the Riccati--type differential identity \cite[p. 268, Eq. (1)]{Remmert} 
   \[ \varepsilon_1'= -\varepsilon_1^2 - \pi^2\, .\]
Also, we observe that \eqref{W5} becomes the Eisenstein series $\varepsilon_1(z)$ for $c_2 = - \frac12$. 
\end{proof}
   
The cotangent form of $\varepsilon_1(z)$ and the examples \eqref{FX} are best expressed and extended when one recalls the 
beautiful reflection formula  for the more-practical 
{\it digamma--function} $\psi(z) := \tfrac{{\rm d}}{{\rm d}z}\, \log \Gamma(z) = \Gamma'(z)/\Gamma(z)$, namely 
\cite[p. 259, Eq. 6.3.7]{AS} 
   \begin{equation} \label{A1}
      \varepsilon_1(z) = \pi \cot(\pi z) = \psi(1-z) - \psi(z)\, ,
   \end{equation}
for which the $r$--th derivative \cite[p. 260, Eq. 6.4.10]{AS}, the so-called {\em polygamma function} reads
   \begin{equation} \label{A11}
      \psi_r(z) = \psi^{(r)}(z) = (-1)^{r+1}\, r! \sum_{ k \in \mathbb N_0}
                  \dfrac1{(z+k)^{r+1}}\, \qquad 
                  \left( z \in \mathbb C \setminus \mathbb Z_0^-,\, 
                  r \in \mathbb N \right).
   \end{equation}
Special attention is given to the case $r=0$, that is
   \begin{equation} \label{D7}
      \psi(z) := \psi_0(z) = \psi^{(0)}(z) = \sum_{k \in \mathbb N} \left( \dfrac1k 
               - \dfrac1{z+k-1}\right) - {\boldsymbol \gamma}\, ,
   \end{equation}
where $\boldsymbol \gamma = 0.5772156649...$ signifies the Euler--Mascheroni constant. 
   
A first new, but simple result in this respect reads, noting \eqref{A20},

\begin{proposition} 
For all $z \in \mathbb C \setminus \mathbb Z$ we have
   \[ \varepsilon_r(z) = \frac1{\Gamma(r)}\big( \psi_{r-1}(1-z) + (-1)^r \psi_{r-1}(z)\big)\] 
\end{proposition}

As to the proof, it follows immediately from \eqref{A1} and \eqref{A20}. 

Our first more important result is a new integral representation of $\varepsilon_r(z)$.

\begin{theorem}
There holds the integral representation
   \[ \varepsilon_r(z) = (z - [\Re(z)])^{-r} + \frac1{\Gamma(r)} 
                         \int_0^\infty \frac{t^{r-1}}{e^t-1} 
                         \big( {\rm e}^{-(z - [\Re(z)])t} 
                       + (-1)^r\,{\rm e}^{(z - [\Re(z)])t}\big)\, {\rm d}t\, ,\]
for all $r \in \mathbb N$ and for all $z\in \mathbb C \setminus \mathbb Z$. Here $[x]$ stands for the integer part of\, $x \in \mathbb R$.
\end{theorem}

\begin{proof} 
By the $1$-periodicity of Eisenstein's functions $\varepsilon_r(z)$, it is sufficient to consider it inside the vertical strip 
$\Re(z) \in (0,1)$ of the complex plane. Indeed, otherwise, assuming $z \neq 0$, by the relation $\varepsilon_r(z) = 
\varepsilon_r(z - [\Re(z)])$, we have the same property. Therefore, letting $\Re(z) \in (0,1)$, by the Gamma--function formula 
   \[ \Gamma(r)\,A^{-s} = \int_0^\infty t^{s-1} {\rm e}^{-At}\, {\rm d}t, \qquad \big( \Re(s)>0, \, \Re(A)>0\big)\, ,\] 
we conclude  
   \begin{align*} 
      \varepsilon_r(z) &= \sum_{k \in \mathbb Z} \frac1{(z+k)^r} = \dfrac1{z^r} + \sum_{k \in \mathbb N} 
                          \left( \frac1{(z+k)^r} + \frac{(-1)^r}{(k-z)^r} \right) \\       
                       &= \dfrac1{z^r} + \dfrac1{\Gamma(r)}\int_0^\infty t^{r-1}
                          \Big(\sum_{k \in \mathbb N} {\rm e}^{-kt}\Big)\,
                          \Big({\rm e}^{-zt} + (-1)^r {\rm e}^{zt}\Big)\, {\rm d}t \\
                       &= z^{-r} + \frac1{\Gamma(r)} \int_0^\infty \frac{t^{r-1}}{1-{\rm e}^{-t}} \big( {\rm e}^{-(z+1)t} 
                        + (-1)^r\,{\rm e}^{-(1-z)t}\big)\, {\rm d}t\, .
   \end{align*}
The integral converges for $r \ge 1$, when $|\Re(z)|<1$, as the integrand's behavior is controlled near to the origin and at infinity. 
The rest is clear. 
\end{proof} 

\begin{corollary}
For all $r \in \mathbb N$ and $z \in \mathbb C \setminus \mathbb Z$, we have 
   \[ \varepsilon_r(z) = \frac1{(z - [\Re(z)])^r} + \frac2{\Gamma(r)}\int_0^\infty \frac{t^{r-1}}{e^t-1} 
                          \left\{ \begin{array}{r} 
                                     \! \cosh(z - [\Re(z)])t \\
                                     \! -\sinh (z - [\Re(z)])t
                                   \end{array} \right\}\, {\rm d}t, \quad
                          \begin{cases} 
                              r\quad {\rm even}\\
                              r\quad {\rm odd}
                           \end{cases}. \]
\end{corollary} 

\section{Backgrounds to Hilbert--Eisenstein series}

A basis to the Hilbert--Eisenstein series includes the classical Bernoulli numbers $B_n := B_n(0), n \in \mathbb N_0$, defined in 
terms of the Bernoulli polynomials $B_n(x)$, defined, for example, {\it via} their  exponential generating function
   \begin{equation} \label{B-1}
	    \sum_{n \in \mathbb N_0} B_n(x) \frac{z^n}{n!} = \dfrac{z{\rm e}^{zx}}{{\rm e}^x-1}\,
              \qquad \big( z \in \mathbb C,\, |z|<2\pi,\, x \in \mathbb R\big)\, .
	 \end{equation}
We need some facts concerned with $B_n(x)$. Starting with the $1$--periodic Bernoulli polynomials $\mathscr B_n(x)$ defined as the 
periodic extension of $\mathscr B_n(x) = B_n(x)$, $x \in (0,1]$, we need to introduce the $1$--periodic conjugate functions 
$\mathscr B_n^\sim(x)$, $x \in \mathbb R$ ($x \not\in \mathbb Z$ if $n=1$) by
   \[ \mathscr B_n^\sim(x) := \mathscr H_1 \left[ \mathscr B_n(\cdot)\right](x), 
                              \qquad (n \in \mathbb N)\, .\]
Here $\mathscr H_1$ is the (periodic) Hilbert transform  of the $1$--periodic function $\varphi$ defined by
   \[ \mathscr H_1[\varphi](x) = {\rm PV}\, \int_{-\frac12}^{\frac12} \varphi(x-u)\, 
                                 \cot (\pi u)\, {\rm d}u\, ,\]
so  that
   \[ \mathscr B_n^\sim(x) = {\rm PV}\, \int_{-\frac12}^{\frac12}\mathscr B_n(x-u) 
                             \cot (\pi u)\, {\rm d}u\, ,\]
with $B_0^\sim(x) = \mathscr B_0^\sim(x)= 0$ for all $x \in \mathbb R$, since $B_0(x) = \mathscr B_0(x)=1$. Written as a Fourier series, 
they then are to be \cite{Butzer}
   \begin{equation} \label{X1}
      \mathscr B_{2n+1}^\sim(x) = -2(2n+1)! \sum_{k\in \mathbb N}\frac{\sin\left(2\pi kx-(2n+1)\frac\pi2\right)}{(2\pi k)^{2n+1}}, 
                                  \quad \big(n \in \mathbb N_0\big).
   \end{equation}
These conjugate periodic functions $\mathscr B_n^\sim(x)$ are used to define the  non--periodic functions $B_n^\sim(x)$, 
which can be regarded as conjugate Bernoulli  "polynomials" in a form such that their properties are similar to those of the 
classical  Bernoulli polynomials $B_n(x)$. For details see Butzer and Hauss \cite[p. 22]{ButzerHauss} and Butzer 
\cite[pp. 37-56]{Butzer}. The conjugate Bernoulli numbers needed, the $B_{2m+1}^\sim$, are the 
$B_{2m+1}^\sim(0) ( = B_{2m+1}^\sim(1))$ for which 
   \[ B_{2m+1}^\sim( \tfrac12) = \big( 2^{-2m}-1\big)\cdot B_{2m+1}^\sim(1), 
                                 \qquad B_{2m}^\sim(\tfrac12) = 0\, .\]
Some values of the conjugate Bernoulli numbers are (see \cite[p. 69]{Butzer})
   \begin{equation} \label{X2}
      B^\sim_{2m+1}\!(\tfrac12) = \begin{cases}
                - \displaystyle \frac{\log2}\pi      &  m = 0\\
                  \displaystyle \frac{\log2}{4\pi} -2 \int_{0+}^{\frac12} u^3 \cot (\pi u) {\rm d}u &  m = 1\\
                  \displaystyle \frac{11}8\int_{0+}^{\frac12} u\cot\pi{\rm d}u + \frac53\int_{0+}^{\frac12} u^3 \cot(\pi u)\,{\rm d}u\\
									\qquad \displaystyle - 2\int_{0+}^{\frac12}u^5\cot (\pi u)\, {\rm d}u &  m = 2
              \end{cases}\quad , 
   \end{equation} 
and
   \[ B^\sim_1\!(x) = -\displaystyle \frac1\pi \, \log \big( 2\sin(\pi x)\big)\, .\]
Of basic importance is also the {\it exponential generating function} of $B_k^\sim(\tfrac12)$, 
given for $|z|<2$ by
   \begin{equation} \label{B0}
      \sum_{k \in \mathbb N_0} B_k^\sim\big(\tfrac12\big)\, \frac{z^k}{k!} = 
           - \frac{z{\rm e}^{\frac z2}}{{\rm e}^z-1}\, \Omega(z) 
           = - \frac{z}{2\sinh \frac z2}\, \Omega(z)\, ,
   \end{equation}
first established by M. Hauss \cite[p. 91--95]{Hauss}, \cite{Hauss1} (see also \cite[pp. 21--29.]{ButzerHauss} and 
\cite[pp. 37--38, 78--80]{Butzer}). Above, $\Omega(z), z \in \mathbb C$ is the so--called Butzer--Flocke-Hauss (complete) Omega function 
introduced in \cite{BFH} in the form
   \[ \Omega(z) := 2 \int_{0+}^{\frac12} \sinh (z u) \cot (\pi u)\, {\rm d}u, 
                   \qquad \big( z \in \mathbb C\big)\, .\]         
It is the Hilbert transform $\mathscr H_1[{\rm e}^{-zx}](0)$ at zero of the $1$--periodic function $\big({\rm e}^{-zx}\big)_1$, 
defined by the periodic extension of the exponential function ${\rm e}^{-zx}$, $|x|< \tfrac12,\, z \in \mathbb C$, thus
   \[ \mathscr H_1\big[{\rm e}^{-zx}\big](0) = {\rm PV}\, \int_{-\frac12}^{\frac12} 
                       {\rm e}^{zu}\, \cot (\pi u)\, {\rm d}u = \Omega(z)\, .\] 
As to the Omega function, we further need its {\it basic partial fraction development} for $z \in \mathbb C \setminus 
{\rm i}\mathbb Z$, namely 
   \begin{equation} \label{B01} 
	    \Omega(2\pi z) = \frac1\pi\, \big({\rm e}^{-\pi z} - {\rm e}^{\pi z}\big) 
                        \sum_{k \in \mathbb N}  \frac{(-1)^k\,k}{z^2+k^2} 
                     = - \frac{{\rm i}\sinh (\pi z)}{\pi}\,
                        \underset{k \in \mathbb Z}{\sum\nolimits_e}\dfrac{(-1)^k \, {\rm sgn}k}{z+{\rm i}k} \, ,
   \end{equation}
the proof of which depends upon a new Hilbert--Poisson formula, introduced by Hauss; see \cite[pp. 97--103]{Hauss} or \cite{Hauss1}. 

A useful formula which will link Hilbert--Eisenstein series, Hilbert transform versions of the Bernoulli numbers and the Riemann 
Zeta function is given by (see \cite[Eq. 1.17(7)]{Erdelyi} and \cite[Eq. (54.10.3)]{Hansen})
   \begin{equation} \label{B2}
      \sum_{k \in \mathbb N}\zeta(2k+1)\, z^{2k}  
           = - \frac 12 \left[\psi(1+z)+ \psi(1-z)\right] + \boldsymbol \gamma, 
             \qquad \big( |z|<1 \big)
   \end{equation}
or, replacing $z \mapsto z{\rm i}$, then 
   \begin{equation} \label{B3}
      \sum_{k \in \mathbb N} (-1)^{k-1} \zeta(2k+1)\, z^{2k}\,  
           =  \frac 12 \left[\psi(1+{\rm i}z)+ \psi(1-{\rm i}z)\right] + 
              \boldsymbol \gamma,  \qquad \big( |z|<1 \big)\, . 
   \end{equation}
A second formula in this respect reads \cite[6.3.17]{AS}, \cite[Eq. (54.3.5)]{Hansen}
   \begin{equation} \label{B4}
      \sum_{k \in \mathbb N} (-1)^{k-1} \zeta(2k+1)\, z^{2k} 
           = \boldsymbol \gamma + \Re\big[ \psi(1+{\rm i}z)\big], 
             \qquad \big( z \in \mathbb R \big)\, . 
   \end{equation} 

\section{Hilbert--Eisenstein series}

Now, we come to the main sections of this article, dealing with Hilbert--Eisenstein series. A Hilbert (conjugate function) -- type 
version of the Eisenstein series $\varepsilon_1(z)$ was first studied by Michael Hauss in his doctoral thesis \cite{Hauss}. 

\begin{definition}
The Hilbert--Eisenstein ${\rm (HE)}$ series $\mathfrak h_r(w)$ are defined for $z \in \mathbb C \setminus {\rm i}\mathbb Z$ and 
$r \in \mathbb N_2$ by
   \[ \mathfrak h_r(z) := \sum_{k \in \mathbb Z} \dfrac{(-1)^k {\rm sgn}(k)}{(z+{\rm i}k)^r} 
                        = \sum_{k \in \mathbb N} (-1)^k \left( \dfrac1{(z + {\rm i}k)^r} 
                        - \dfrac1{(z - {\rm i}k)^r}\right)\, ,\]
and, for $r = 1$ recalling \eqref{B01}, by
   \[ \mathfrak h_1(z) := \underset{k \in \mathbb Z}{\sum\nolimits_e} 
                          \dfrac{(-1)^k {\rm sgn}(k)}{z+{\rm i}k} 
                        = \dfrac{{\rm i}\pi\, \Omega(2\pi z)}{\sinh \pi z}
                        = {\rm i}\, \Omega(2\pi z) \sum_{k \in \mathbb Z} \dfrac{(-1)^k}{z+{\rm i}k} \,,\]
with $\mathfrak h_1(0) = 2{\rm i} \log 2$, noting ${\rm sgn}(0) = 0$. 
\end{definition} 

In this respect recall that
   \[ \dfrac\pi{\sinh(\pi z)} = \sum_{k \in \mathbb Z} \frac{(-1)^k}{z+{\rm i}k} 
                   = \dfrac1z + 2z \sum_{k \in \mathbb N} \frac{(-1)^k}{z^2+k^2}\, ,
									   \qquad \big( z \in \mathbb C \setminus {\rm i}\mathbb Z\big).\]
The basic properties of $\mathfrak h_r$ for $z \in \mathbb C \setminus {\rm i}\mathbb Z$ and $r\in \mathbb N_2,s \in \mathbb N$, are
   \begin{align} \label{DP}
      \mathfrak h'_r(z) &= - r\, \mathfrak h_{r+1}(z)  \nonumber \\
      \mathfrak h_r^{(s)}(z) &= (-1)^s\, (r)_s\, \mathfrak h_{r+s}(z)\, ,
   \end{align}
as well as their difference and symmetry property \cite[Eq. (6.5.72)]{Hauss}, \cite[Eq. (9.7)]{Butzer}
   \[ \mathfrak h_r(z) + \mathfrak h_r(z+ {\rm i}) = z^{-r} - (z+ {\rm i})^{-r}; \quad
	    \mathfrak h_r(-z) = (-1)^{r+1}\mathfrak h_r(z), \qquad \big( z \in \mathbb C \setminus {\rm i}\mathbb Z\big). \]
Above 
   \[ (\rho)_\sigma := \dfrac{\Gamma(\rho+\sigma)}{\Gamma(\rho)} 
                     = \begin{cases}
                          1, & \qquad \big(\sigma = 0;\, \rho \in \mathbb C\setminus \{0\}\big) \\
                          \rho(\rho+1) \cdots (\rho+\sigma-1) & \qquad \quad \big(\sigma \in \mathbb N;\,\, \rho \in \mathbb C\big)
                       \end{cases}\, ;\]
stands for the Pochhammer symbol (or shifted, rising factorial). Note that it being understood conventionally that $(0)_0 := 1$.

\begin{proposition} 
\noindent ${\bf a)}$ For $z \in \mathbb C$, $|z|<1$ one has 
   \begin{equation} \label{X4}  
	    \underset{k \in \mathbb Z}{\sum\nolimits_e} \frac{(-1)^{k}\, {\rm sgn}(k)}{z+{\rm i}k} 
			     = 2{\rm i} \sum_{n \in \mathbb N_0}(-1)^n\, \eta(2n+1) z^{2n}\, .
	 \end{equation}
\noindent ${\bf b)}$ Moreover
   \begin{align} \label{X5}
	    B_{2n+1}^\sim(\tfrac12) &= (-1)^{n+1} (2n+1)! 2^{-2n} \pi^{-2n-1}\, \eta(2n+1) \\	 \label{X6}
	                            &= (-1)^{n} (2n+1)! \big(4^{-2n}-2^{-2n}\big) \pi^{-2n-1}\, \zeta(2n+1)\, ,
	 \end{align}
where $\eta(s) = \sum_{n \in \mathbb N} (-1)^{n-1}n^{-s},\, \Re(s)>0$ stands for the {\em Dirichlet's Eta function}.
\end{proposition}

\begin{proof} We have
   \begin{align*}
	    \lim_{N \to \infty} & \sum_{|k| \leq N} \frac{(-1)^{k}\, {\rm sgn}(k)}{z+{\rm i}k} \\
			  &= \lim_{N \to \infty} \sum_{k=1}^N (-1)^{k-1}\, \left\{ \frac1{z-{\rm i}k} - \frac1{z+{\rm i}k}\right\} \\
				&= 2{\rm i} \sum_{k\in \mathbb N} \frac{(-1)^{k-1}\, k}{z^2+k^2}
				 = 2{\rm i} \sum_{k\in \mathbb N} \frac{(-1)^{k-1}}k \sum_{n\in \mathbb N_0}(-1)^n\Big(\frac zk\Big)^{2n} \\
			  &= 2{\rm i} \lim_{N \to \infty} \sum_{k=1}^N \frac{(-1)^{k-1}}{k} + 2{\rm i} 
				   \sum_{n \in \mathbb N} (-1)^n \left\{ \sum_{k \in \mathbb N} \frac{(-1)^{k-1}}{k^{2n+1}} \right\}\, z^{2n}\\
				&= 2{\rm i} \sum_{n \in \mathbb N_0} (-1)^n \eta(2n+1)\, z^{2n}\,,
   \end{align*}
the interchange of the summation order being possible on account of the Weierstra{\ss} double series theorem (see e.g. 
\cite[p. 428]{Knopp}). This proves part {\bf a)}. \medskip

As to part {\bf b)}, on account of \eqref{B0} 
   \[ \frac1{2z}\, \sum_{n \in \mathbb N_0} B_n^\sim\big(\tfrac12\big)\, \frac{(2\pi z)^n}{n!} = 
           \frac{\pi\, \Omega(2\pi z)}{{\rm e}^{-\pi z}-{\rm e}^{\pi z}}\,.  \]
Comparing coefficients with \eqref{X4}, gives $B_{2n}^\sim\big(\tfrac12\big) = 0$ and results in \eqref{X5}. As to \eqref{X6}, 
it follows from	\eqref{X5} by noting that $\eta(2n+1) = \big(1-2^{-2n}\big)\, \zeta(2n+1)$, $n\in \mathbb N$. 

Alternatively, \eqref{X6} follows from \eqref{X1}, by setting $x=\tfrac12$, which yields 
   \[ B_{2n+1}^\sim\big(\tfrac12\big) = 2(2n+1)! \sum_{k\in \mathbb N}\frac{\sin\big((k+\tfrac12)\pi\big)}{(2\pi k)^{2n+1}} 
	                            = (-1)^n \frac{2(2n+1)!}{(2\pi)^{2n+1}}\, \eta(2n+1)\, .\]
This finishes the proof of proposition. 
\end{proof} 

Now, the generating function of $B_k^\sim(\tfrac12)$ can be expressed in terms of the digamma function. In fact, 

\begin{theorem}
For $z \in \mathbb C \setminus (1+{\rm i})\mathbb Z$ with $|z|<2\pi$ there holds
   \begin{align} \label{X3} 
      &\sum_{k \in \mathbb N_0} B_k^\sim(\tfrac12) \frac{z^k}{k!}\nonumber \\ 
			   &\quad = \begin{cases}
            - \dfrac z\pi\, \Big\{\log2 + \psi\Big(1+\dfrac{{\rm i}z}{4\pi}\Big) 
            - \psi\Big(1+\dfrac{{\rm i}z}{2\pi}\Big)\Big\} \\
            \qquad + \dfrac{{\rm i}z}{2\pi}\Big\{ \coth\dfrac{z}{4\pi} - \coth\dfrac{z}{2\pi}\Big\} -{\rm i}\,,   
						& \big(|z| <2\pi\big)\\ \\
            - \dfrac z\pi \Big\{ \log2 + \Re\, \psi\Big(1+\dfrac{{\rm i}z}{4\pi}\Big)
						- \Re\, \psi\Big(1+\dfrac{{\rm i}z}{2\pi}\Big) \Big\}\,, 
						& (-2\pi<z< 2\pi)             
         \end{cases}.
   \end{align}
\end{theorem}

\begin{proof} 
Substituting formula \eqref{X6} for $B^\sim_{2m+1}\!(\tfrac12)$ into the series below, and observing \eqref{B2}, we have
   \begin{align*}
      \sum_{k \in \mathbb N_0} B_k^\sim(\tfrac12) \, \frac{z^k}{k!} 
           &= - \frac{\log 2}\pi\cdot z
              + \sum_{k \in \mathbb N} B_k^\sim(\tfrac12) \, \frac{z^k}{k!}\\
           &= - \frac{\log 2}\pi\cdot z + 4 \sum_{k \in \mathbb N} (-1)^k 
              \Big( \frac z{4\pi}\Big)^{2k+1}\, \zeta(2k+1)\\ 
      &\qquad  - 2\sum_{k \in \mathbb N}(-1)^k \Big(\frac z{2\pi}\Big)^{2k+1}\,\zeta(2k+1) \\
           &= - \frac{\log 2}\pi\cdot z + 4 \Big\{ \frac z{8\pi} \Big[ - 2\boldsymbol \gamma 
              - \psi\Big( 1- \frac{{\rm i}z}{4\pi}\Big) 
              - \psi\Big( 1+ \frac{{\rm i}z}{4\pi}\Big)\Big]\Big\} \\
      &\qquad - 2\Big\{\frac z{4\pi}\Big[ - 2\boldsymbol \gamma 
              - \psi\Big( 1- \frac{{\rm i}z}{2\pi}\Big) 
              - \psi\Big( 1+ \frac{{\rm i}z}{2\pi}\Big)\Big]\Big\}\,.
   \end{align*}
This establishes the representation in \eqref{X3} for complex $|z|< 2\pi$. That for real $z \in (-2\pi, 2\pi)$ follows from \eqref{B4}. 
\end{proof}

Observe that the proof of Theorem 3.3 has the same outward appearance as that of Theorem 7.3. in \cite[p. 74]{Butzer}, but it uses the 
correct formula \eqref{X2}, provided with two proofs in Proposition 3.2. 

Now, the $\mathfrak h_1(z)$ can also be represented in terms of the classical digamma function.

\begin{theorem}      
There holds
   \[ \mathfrak h_1(z) = \left\{ 
             \begin{array}{lc} 
                  2{\rm i} \log 2 + {\rm i} \Big\{ \psi\Big(1 + {\rm i}\dfrac z2 \big)
                + \psi\Big(1 - {\rm i}\dfrac z2 \Big) \\ 
                  \qquad \qquad \quad - \psi\big(1 + {\rm i}z \big) 
                - \psi\big(1-{\rm i}z\big)\Big\} &  \big( |z|<2\pi \big) \\   \\
                  2{\rm i} \log 2 + 2{\rm i} \Re \Big\{ \psi\Big(1+{\rm i} \dfrac z2\Big) 
                - \psi\Big(1+{\rm i}z \Big)\Big\} & \quad (-2\pi<z<2\pi) 
             \end{array} \right. \, .\]
\end{theorem} 

\noindent {\bf First proof.} 
According to \eqref{X4} of Proposition 3.2 and Definition 3.1, we have, noting $\eta(s) = (1-2^{1-s})\zeta(s)$, $\Re(s)>1$,
   \begin{align*}
       \mathfrak h_1(z) &= 2{\rm i}\, \sum_{ k \in \mathbb N_0}(-1)^k\, \eta(2k+1)\, z^{2k}\\
			        &= 2{\rm i}\,\sum_{k \in \mathbb N_0} (-1)^k\,
                 \big(1-2^{-2k}\big)\, \zeta(2k+1)\, z^{2k}\\
              &= 2{\rm i}\,\sum_{k \in \mathbb N_0} \left( ({\rm i}z)^{2k} 
               - \big(\tfrac{{\rm i}z}2\big)^{2k}\right) \, \zeta(2k+1) \\
              &= 2{\rm i}\,\Big\{ \lim_{h \to 0_+}\left( ({\rm i}z)^{2h} 
                         - \big(\tfrac{{\rm i}z}2\big)^{2h}\right) \, \zeta(2h+1)\\ 
              &\qquad \quad + \sum_{k \in \mathbb N} \left( ({\rm i}z)^{2k} 
               - \big(\tfrac{{\rm i}z}2\big)^{2k}\right)\, \zeta(2k+1)\Big\} \, .
   \end{align*}
We now express the sum {\it via} the linear combination of digamma functions, recalling \eqref{B3}, that means 
   \begin{align*}
       \mathfrak h_1(z) &= 2{\rm i} \,\Bigg\{ \lim_{h \to 0_+}\left( ({\rm i}z)^{2h} 
                - \left(\frac{{\rm i}z}2\right)^{2h}\right)\, \zeta(2h+1) \\
               & \qquad + \frac 12 \Big[\psi\left(1+ {\rm i}\frac z2\right) 
                + \psi\left(1- {\rm i}\frac z2\right)\Big] 
                - \frac 12 \big[\psi\left(1+ {\rm i}z\right) 
                + \psi\left(1- {\rm i}z\right)\big] \Bigg\} \, .
   \end{align*}
On the other hand 
   \begin{align*}
      \lim_{h \to 0_+}\Bigg(({\rm i}z)^{2h} 
            - \left(\frac{{\rm i}z}2\right)^{2h}\Bigg) \zeta(2h+1) 
           &= \lim_{h \to 0_+}\Bigg( ({\rm i}z)^{2h} - \left(\frac{{\rm i}z}2\right)^{2h}\Bigg)
              \left( \frac1{2h} + \boldsymbol \gamma + o(h) \right) \\
           &= \log ({\rm i}z) - \log \big( \tfrac{{\rm i}z}2\big) = \log 2 \, .
    \end{align*}
Now, obvious steps lead to the assertion of Theorem 3.4. \hfill $\Box$ \medskip 

Observe that the real parts of Theorem 3.4 can also be expressed as integrals, noting 
   \[ \Re\, \psi\Big(1+\frac{{\rm i}z}{2\pi}\Big) = - \boldsymbol \gamma + 
	        2\int_0^\infty {\rm e}^{-u}\frac{\sin^2\left( \frac{zu}{2\pi}\right)}{\sinh(u)}\,{\rm d}u\, .\]
Although Theorem 3.4 is to be found in \cite[Eq. (7.8)]{Butzer}, the above proof is a new approach to Hilbert--Eisenstein series. 
\medskip

\noindent {\bf Second proof.} According to \eqref{B0}, we have on the one hand
   \begin{equation} \label{C5}
      \sum_{k \in \mathbb N_0} B_k^\sim\big(\tfrac12\big)\, \frac{(2\pi z)^k}{k!} 
              = - \frac{\pi z}{\sinh (\pi z)}\, \Omega(2\pi z), \qquad \big( z \in \mathbb C \setminus {\rm i}\mathbb Z\big)\,,
   \end{equation}
and, on the other hand
   \begin{equation} \label{C6}
      \mathfrak h_1(z) = \frac{{\rm i}\pi}{\sinh (\pi z)}\, \Omega(2\pi z)\, .
   \end{equation}
Thus, following the argument along the lines of the proof of Theorem 3.3, 
   \begin{align*}
      -\frac z{\rm i}\, \mathfrak h_1(z) 
           &= \sum_{k \in \mathbb N_0} B_k^\sim\big(\tfrac12\big)\, \frac{(2\pi z)^k}{k!} \\
           &= -2z\, \log 2 + 2 \sum_{k \in \mathbb N} (-1)^k 
              \Big( \frac{2\pi z}{4\pi}\Big)^{2k+1} \zeta(2k+1)\\ 
    &\qquad - \sum_{k \in \mathbb N} (-1)^k \Big( \frac{2\pi z}{2\pi}\Big)^{2k+1} \zeta(2k+1)\\
           &= -2z\, \log 2 + 2 \cdot \frac z4\, \Big\{ 2 \psi(1) 
            - \psi\Big(1 + {\rm i}\frac z2 \Big)- \psi\Big(1 - {\rm i}\frac z2 \Big)\Big\}\\
    &\qquad - \frac z2\, \big\{ 2 \psi(1) - \psi\big(1 + {\rm i}z\big) 
            - \psi\big(1 - {\rm i}z \big)\big\}\, .
   \end{align*}  
Therefore 
   \[ \mathfrak h_1(z) = 2{\rm i}\, \log 2 + {\rm i}\, \big\{ \psi\big(1 + {\rm i}\tfrac z2 \big)
                             + \psi\big(1 - {\rm i}\tfrac z2 \big) 
                             - \psi\big(1 + {\rm i}z \big) - \psi\big(1-{\rm i}z\big)\big\}\, .\]
This completes the proof of the first formula of Theorem 3.4. The second one follows immediately by the 
mirror symmetry formula $\psi(\overline w) = \overline{\psi(w)},\, w \in \mathbb C$.  

It is important to mention that this representation of $\mathfrak h_1(z)$ is not given in \cite{Butzer}, but contained 
implicitly in a more complicated form in the proof of Proposition 6.4.1. in \cite{Hauss1}.  \hfill $\Box$
  
\begin{corollary}
The Omega function $\Omega(z)$ has the representation 
   \[ \Omega(z) = \frac1\pi \sinh \Big( \frac{z}{2}\Big) \Big\{ 2\log 2 
                + \psi\Big( 1+\frac{{\rm i}z}{4\pi}\Big) + \psi\Big( 1- \frac{{\rm i}z}{4\pi}\Big)
                - \psi\Big( 1+\frac{{\rm i}z}{2\pi}\Big) - \psi\Big( 1- \frac{{\rm i}z}{2\pi}\Big) \Big\} ,\]
for $z \in \mathbb C \setminus \mathbb Z_0^-, |z|< 2\pi$. 
\end{corollary}

The proof is immediate from Theorem 3.3 in view of 
   \[ \Omega(z) = - \frac{\rm i}\pi \sinh \Big( \frac{z}2\Big) \cdot \mathfrak h_1 \Big( \frac{z}{2\pi}\Big)\, ,\] 
so it is omitted. However, we remark that this corollary could also be derived, just as simply, {\it via} Theorem 3.4. 

Although, as observed in \cite[p. 67]{Butzer}, the Omega function is not an "elementary function", it is nevertheless expressible 
in terms of the hyperbolic sine function multiplied by a (simple) linear combination of digamma functions.

\section{A novel alternative approach to Theorem 3.4}

The representation of $\mathfrak h_1(z)$ in terms of certain combinations of $\psi$--functions with the constant $2{\rm i}\, \log 2$ 
(Theorem 3.4), was established {\em via} the generating function of the conjugate Bernoulli numbers $B_k^\sim\big( \tfrac12\big)$, 
equation \eqref{B0}, plus arguments used in the proof of Theorem 3.3, which in turn were based fundamentally upon the representation 
of $\zeta(2k+1)$ in terms of $B_{2k+1}^\sim(0)$, thus 
   \[ \zeta(2m+1) = (-1)^{m} 2^{2m}\pi^{2m-1}\, \frac{B_{2m+1}^\sim (0)}{(2m+1)!}, 
                    \qquad \big( m \in \mathbb N\big)\,,\]
as well as the delicate formulae \eqref{B3} and \eqref{B4} given in tables by Hansen \cite{Hansen} and Abramowitz--Stegun \cite{AS}. 
But the last four formulae are only to be found in tables. Thus one does not know of the possible difficulties of their proofs. 

A further aim of this article is to present alternative proofs which are fully independent of these two formulae \eqref{B3} and 
\eqref{B4}. Moreover, since two proofs of Theorem 3.4 presented are based on power series expansions in the open unit disc $|z|<1$, 
we can extend the validity range to the whole $\big( \mathbb C \setminus {\rm i}\mathbb Z\big) \cup \{0\} $ by our present approach. 

There exists a vast literature concerning Mathieu series $S_r(x)$ and the more recent alternating Mathieu series $\widetilde S_r(x)$, 
both of which are defined by 
   \begin{align} \label{D1}
      S_r(x) &= \sum_{k \in \mathbb N} \frac{2k}{(k^2+x^2)^r}, 
                \qquad \big( r>1 \big)\\ \label{D2}
      \widetilde S_r(x) &= \sum_{k \in \mathbb N} \frac{(-1)^{k-1}2k}{(k^2+x^2)^r}, 
                \qquad \big( r>0 \big)\, ;
   \end{align}
see among others \cite{JunSri}, \cite{PST}, \cite{PTL} and the references therein. These articles are of interest, in particular since 
the series $\widetilde S_r$ is connected to HE series $\mathfrak h_1$. We now come to a {\it new third proof} of Theorem 3.4, 
at the same time establishing the representation not only for $\mathfrak h_1$ but also for higher order $\mathfrak h_r$ on 
the extended range. 

\begin{theorem}
For all $z \in \big(\mathbb C \setminus {\rm i}\mathbb Z\big) \cup \{0\}$ there holds
   \begin{equation} \label{D3}
      \mathfrak h_1(z) = 2{\rm i} \log 2 + {\rm i} \Big\{ \psi\Big(1+{\rm i} \dfrac z2\Big) 
                       + \psi\Big(1-{\rm i} \dfrac z2\Big) - \psi\Big(1+{\rm i}z \Big) - \psi\Big(1-{\rm i}z \Big) \Big\}\, . 
   \end{equation}
Moreover, for the same $z$--domain we have for $r \in \mathbb N_2$
   \begin{align} \label{D5}
      \mathfrak h_r(z) &= \dfrac{{\rm i}^r}{\Gamma(r)}\, \Big\{ 2^{-r+1}\,
                    \psi_{r-1}\left(1 - {\rm i}\frac z2 \right) 
                  + (-2)^{-r+1}\,\psi_{r-1}\left(1+{\rm i}\frac z2 \right)\nonumber\\ 
          &\qquad - \psi_{r-1}\left(1 - {\rm i}z \right) 
                  - (-1)^{r-1}\psi_{r-1}\left(1 + {\rm i}z\right)\Big\}\, .
    \end{align} 
\end{theorem}

\begin{proof} 
To perform the proof of the representation formula \eqref{D3} we have to connect the Hilbert--Eisenstein series $\mathfrak h_1(z)$, 
which is to be understood in the sense of Eisenstein  summation, with the normally convergent HE series $\mathfrak h_r(z), r \ge 2$ 
(for the latter see Remmert \cite[p. 290]{Remmert}), which is termwise integrable (see \cite[p. 42]{Remmert1}).  

From Definition 3.1 and the series form  \eqref{D7} of the digamma function, using the straightforward representation of the alternating 
series, say
   \[ \sum_{k \in \mathbb Z} (-1)^{k-1} a_k = \sum_{k\,\,{\rm odd}}a_k - \sum_{k\,\,{\rm even}}a_k = \sum_{k \in \mathbb Z}a_k 
                                              - 2\sum_{k \in \mathbb Z}a_{2k}\,,\]
we have for all $z \in \big(\mathbb C \setminus {\rm i}\mathbb Z\big) \cup \{0\}$,
   \begin{align} \label{C41}
      \mathfrak h_2(z) &= - \sum_{k \in \mathbb N_0} \left( \dfrac1{(z + {\rm i}k)^2} 
                   - \dfrac1{(z - {\rm i}k)^2}\right)  
                   + 2\, \sum_{k \in \mathbb N_0} \left( \dfrac1{(z + 2{\rm i}k)^2} 
                   - \dfrac1{(z - 2{\rm i}k)^2}\right) \nonumber \\
                  &= -\sum_{k \in \mathbb N_0} \left( \dfrac1{(k+{\rm i}z)^2} 
                   - \dfrac1{(k - {\rm i}z)^2}\right)
                   + \frac12\sum_{k\in \mathbb N_0} \left( \dfrac1{(k+{\rm i}\frac z2)^2} 
                   - \dfrac1{(k - {\rm i} \frac z2)^2}\right)\nonumber \\
                  &= \frac12 \left[ \psi_1\left(1 + {\rm i} \frac z2\right) 
                   - \psi_1\left(1 - {\rm i} \frac z2\right)\right] 
                   - \psi_1\left(1 + {\rm i}z\right) + \psi_1\left(1 - {\rm i} z\right)\, ; 
   \end{align} 
actually, we employ here the {\em trigamma function} $\psi_1(z) = \sum_{k \in \mathbb N_0}(k+z)^{-2}$, which normally converges 
in $\big(\mathbb C \setminus {\rm i}\mathbb Z\big) \cup \{0\}$. 

Term--wise integration then implies 
   \[ \int_0^z \mathfrak h_2(t)\, {\rm d}t = \mathfrak h_1(0) - \mathfrak h_1(z) 
             =  2{\rm i}\, \log 2 - \mathfrak h_1(z) \, .\]
Integrating \eqref{C41} directly on $[0,z]$ too, we obtain
   \[ \int_0^z \mathfrak h_2(t)\, {\rm d}t = -{\rm i} \left[ \psi\left(1 
               + {\rm i} \frac z2\right) 
               + \psi\left(1 - {\rm i} \frac z2\right) - \psi\left(1 + {\rm i}z\right) 
               - \psi\left(1 - {\rm i} z\right) \right]\, ,\] 
which completes the proof of the desired closed form representation \eqref{D3}. 

In view of the differentiation property \eqref{DP} of the HE series (see also \cite[p. 796, Eq.  (41)]{BBP}), 
   \begin{equation} \label{C411}
	    \mathfrak h_r(z) = \dfrac{(-1)^r}{\Gamma(r)}\, \mathfrak h_2^{(r-2)}(z),
	 \end{equation}
where $z \in \big(\mathbb C \setminus {\rm i}\mathbb Z\big) \cup \{0\}$ and $r\in \mathbb N_2$, applied to \eqref{C41}, we obtain
   \begin{align*} 
      \mathfrak h_r(z) &= \dfrac{{\rm i}^r}{\Gamma(r)}\, \Big\{ 2^{-r+1}\,
                          \psi_{r-1}\left(1 - {\rm i}\frac z2 \right) 
                        + (-2)^{-r+1}\,\psi_{r-1}\left(1+ {\rm i}\frac x2 \right) \nonumber\\ 
                &\qquad - \psi_{r-1}\left(1 - {\rm i}x \right) 
                        - (-1)^{r-1}\psi_{r-1}\left(1 + {\rm i}x\right)\Big\}\, ,
    \end{align*} 
which completes the proof of the theorem. 
\end{proof}

The restriction of \eqref{D5} to $\mathbb R$ yields

\begin{corollary}
For all $x \in \mathbb R, r \in \mathbb N$ we have 
   \[ \mathfrak h_r(x) = \begin{cases}
               \dfrac{{\rm i}(-1)^{\frac{r-1}2}}{\Gamma(r)}\, \Re\left\{ 2^{-r+1}\, 
               \psi_{r-1}\left(1 + {\rm i}\tfrac x2 \right)
             - \psi_{r-1}\left(1 + x \right)\right\}, & \quad r \quad {\rm odd}\\  \\
               \dfrac{{\rm i}(-1)^{\frac r2-1}}{\Gamma(r)}\, \Im\left\{ 2^{-r+1}\, 
               \psi_{r-1}\left(1 + {\rm i}\tfrac x2 \right)
             - \psi_{r-1}\left(1 + x \right)\right\}, & \quad r \quad {\rm even}
           \end{cases}\, .\]
\end{corollary}\medskip

Finally, let us observe that the HE--series may also be connected with the original Eisenstein series, the  digamma function 
being the connecting link. In fact, for real $z=x$ (not possible for $z$ complex, because otherwise we cannot exploit the mirror 
symmetry formula for the digamma function, that is $\psi(1-{\rm i}z) = \psi\big(\overline{1+{\rm i}z}\big) = 
\overline{\psi(1+{\rm i}z)}$\,)

\begin{theorem}
\noindent ${\bf a)}$ For all $x \in \mathbb R \setminus \{0\}$ 
   \begin{equation} \label{D51}
      \mathfrak h_1(x) = 2{\rm i}\, \log 2 + 2{\rm i} \Re\Big\{\varepsilon_1\left({\rm i} 
                        \frac x2\right) 
                       - \varepsilon_1\big({\rm i} x\big) - \psi\left({\rm i} \frac x2\right) 
                       + \psi\big({\rm i} x\big)  \Big\}\, .
   \end{equation}
\noindent ${\bf b)}$ Also, we have
   \begin{equation} \label{X7}
      \mathfrak h_1(x) = 2{\rm i}\, \log 2 + 2{\rm i} \Re\Big\{\coth\left(\frac x2\right) 
                       - \coth(x) - \psi\left({\rm i} \frac x2\right) + \psi\big({\rm i} x\big)  \Big\}\, .
   \end{equation}
\noindent ${\bf c)}$ For all $x \in \mathbb R \setminus \{0\}$ and $r \in \mathbb N_2$ we have
   \begin{align*} 
      \mathfrak h_r(x) &= {\rm i}^r \left\{ 2^{-r+1}\left( 
                     \varepsilon_r\big({\rm i}\tfrac x2\big) 
                   + (-1)^{r-1} \varepsilon_r\big(-{\rm i}\tfrac x2\big)\right) 
                   - \varepsilon_r({\rm i}x) - (-1)^{r-1} \varepsilon_r(-{\rm i}x)\right\} \\
           &\qquad + \dfrac{(-1)^{r-1}{\rm i}^r}{\Gamma(r)} \,
                     \left\{ 2^{-r+1}\left( \psi_{r-1}\big({\rm i}\tfrac x2\big) 
                   + (-1)^{r-1} \psi_{r-1}\big(-{\rm i}\tfrac x2\big)\right) \right. \\ 
           &\qquad - \left. \psi_{r-1}({\rm i}x) - (-1)^{r-1} \psi_{r-1}(-{\rm i}x)\right\}\,.
   \end{align*}
\end{theorem}

\begin{proof}
The reflection formula \eqref{A1} gives an efficient tool connecting Eisenstein and Hilbert--Eisenstein series. Indeed, replacing 
here $\pm {\rm i}x/2, \pm {\rm i}x$ for $z$, the asserted relation \eqref{D51} follows from Theorem 4.1, \eqref{D3}. 

Now, by the differentiation property \eqref{C411} and \eqref{D5} we connect the HE series $\mathfrak h_r(z)$ and the Eisenstein series 
$\varepsilon_r(z)$, yielding part {\bf c)} for $r \in \mathbb N_2$. 
\end{proof} 

\section{The $\Omega(z)$--function and its properties} 

This section devoted to the $\Omega$--function begins with the cases $r=1$ and $r=2$ of Definition 3.1 thus 
   \begin{align} \label{G1}
	    \frac{{\rm i}\pi \Omega(2\pi z)}{\sin(\pi z)} &= \mathfrak h_1(z), \qquad \big( z\in \mathbb C \setminus {\rm i}\mathbb Z\big) \\
			\mathfrak h_2(z) &= 2{\rm i}z \sum_{k \in \mathbb N} \frac{(-1)^{k-1}2k}{(k^2+z^2)^2} 
			                  = 2{\rm i}z \widetilde S_2(z), \nonumber
   \end{align}
the latter following directly from its definition, having in mind, and noting that the $\mathfrak h_2(z)$ and $\widetilde S_2(z)$ 
are connected {\em via} $\mathfrak h_2(z) = 2{\rm i}z \widetilde S_2(z)$ (see \eqref{D2} as well). An immediate consequence of 
Theorem 4.1 or also of Corollary 3.5 is 

\begin{proposition} 
For all $x \in \mathbb R$, $\Omega(x)$ has the representation 
   \begin{align*} 
	    \Omega(x) &= \frac1\pi\,\sinh\left(\frac x2\right)\,\Big\{ 2 \log2 + \psi\left(1 + {\rm i}\frac x{4\pi} \right) 
                 + \psi\left(1 - {\rm i}\frac x{4\pi} \right) \\
								&\qquad - \psi\left(1 + {\rm i}\frac x{2\pi} \right) - \psi\left(1 - {\rm i}\frac x{2\pi} \right) \Big\}.
   \end{align*}
\end{proposition} 

The next basic property, essentially stated in \cite{BBP}, concerns $\Omega(x)$ as a solution of ODE. 

It would be of great interest to know whether one can express $\psi({\rm i}x) - \psi({\rm i}\tfrac x2)$ in terms of the coth--function 
(or any another hyperbolic function). 

\begin{theorem}
For all $x \in \mathbb R$ the $\Omega(x)$ function is a particular solution of the following linear ${\rm ODE}$:
   \begin{equation} \label{G2}
	    y\,' = \dfrac12\, \coth\left( \frac x2\right)\,y - \dfrac{x}{\pi^3}\, \sinh\left( \frac x2\right)\, E(x)\, ,
	 \end{equation}
where
   \[ E(x) = \begin{cases}
                \widetilde{S}_2(x) = \dfrac1{2x} \displaystyle \int_0^\infty 
								\frac{u\,\sin(xu)\,{\rm d}u}{{\rm e}^u+1}, & \qquad x \neq 0 \\ \\
                2\, \eta(3),        & \qquad x = 0
             \end{cases}\, .\]
\end{theorem}

\begin{proof} Differentiating $\mathfrak h_1(x)$ (or which is the same $\Omega(2\pi x)$) of \eqref{G1} with respect to 
$x \neq 0$, this results in 
   \[ \Omega'(2\pi x) - \frac12 \coth(\pi x)\, \Omega(2\pi x) 
           = -\,\frac{2x \sinh(\pi x)}{\pi^2\, }\,\widetilde{S}_2(x)\, .\]
Substituting $x \mapsto 2\pi x$ we confirm \eqref{G2} for $x \neq 0$. It remains to prove the case $x = 0$, which follows 
by continuity argument is equivalent to the asserted ODE, because 
   \[ E(0) = \lim_{h \to 0}\widetilde S_2\left( h \right) = 2\, \eta(3)\, .\] 
The integral form of $E(x)$ is already reported e.g. in \cite[Eq. (2.8)]{PST}. 
\end{proof}

\begin{theorem} \rm{\cite{BPS}}\, ${\bf a)}$ For all $x \geq 0$ the following two--sided bounding inequalities hold true:
   \[ \frac1\pi \sinh \Big( \frac x2\Big)\, \log\bigg(\frac{\zeta(3)x^2+8\pi^2}{3x^2+2\pi^2}\bigg) 
            \leq \Omega(x) \leq \frac1\pi \sinh \Big( \frac x2\Big)\, \log\bigg(\frac{3 x^2+8\pi^2}{\zeta(3)x^2+2\pi^2}\bigg)\, . \]
Moreover, for $x<0$ the two--sided inequality is reversed. \medskip  

\noindent ${\bf b)}$ For the asymptotic behavior of $\Omega(x)$ for large values of $x$ we have 
   \[ \left( \frac1{2\pi}\, \log\frac{\zeta(3)}{3}\right) \, {\rm e}^{\frac x2} \le \Omega(x) \le \left(\frac1{2\pi}\,  
                  \log\frac3{\zeta(3)}\right) \, {{\rm e}^{\frac x2}},  \qquad \big( x  \to \infty\big)\, .
									\medskip \]
\end{theorem}

See Figure 1 for the graphs of part {\bf a)}.

\begin{figure}[!ht]
    \begin{center}
        \includegraphics[width=0.64\textwidth]{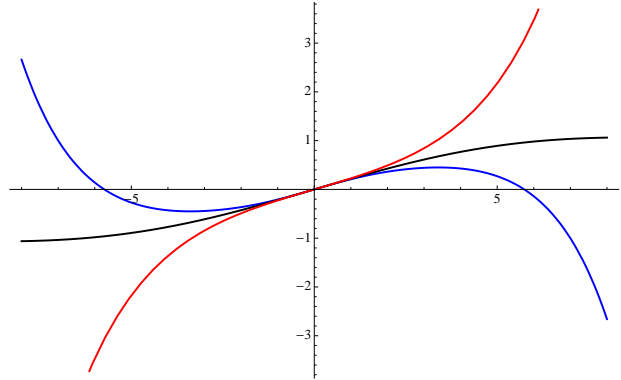}    
    \caption{For $\Omega(x)$ (Theorem 5.3. part {\bf a)}) with $x \in [-8,8]$; for $x>0$ upper bound -- red, lower 
		bound -- blue, $\Omega(x)$ -- black; for $x<0$ {\it vice versa}.}
		\end{center} 
\end{figure} 

See Figure 2 where graph of $\Omega(x)$ lies within horn--type bounds. As to the proof of {\bf b)}, also announced in 
\cite[Theorem 4]{BPS}, we only have to apply the asymptotic of the lower bound in {\bf a)}:  
   \[ \sinh \Big( \frac x2\Big)\, \log\bigg(\frac{\zeta(3)x^2+8\pi^2}{3x^2+2\pi^2}\bigg) 
	          \sim {\rm e}^{\frac x2}\, \log \frac{\zeta(3)}3, \qquad (x \to \infty);\]			
the same procedure leads to the upper bound in Theorem 5.3. {\bf b)}. 

\begin{figure}[!ht]
    \begin{center}
        \includegraphics[width=0.80\textwidth]{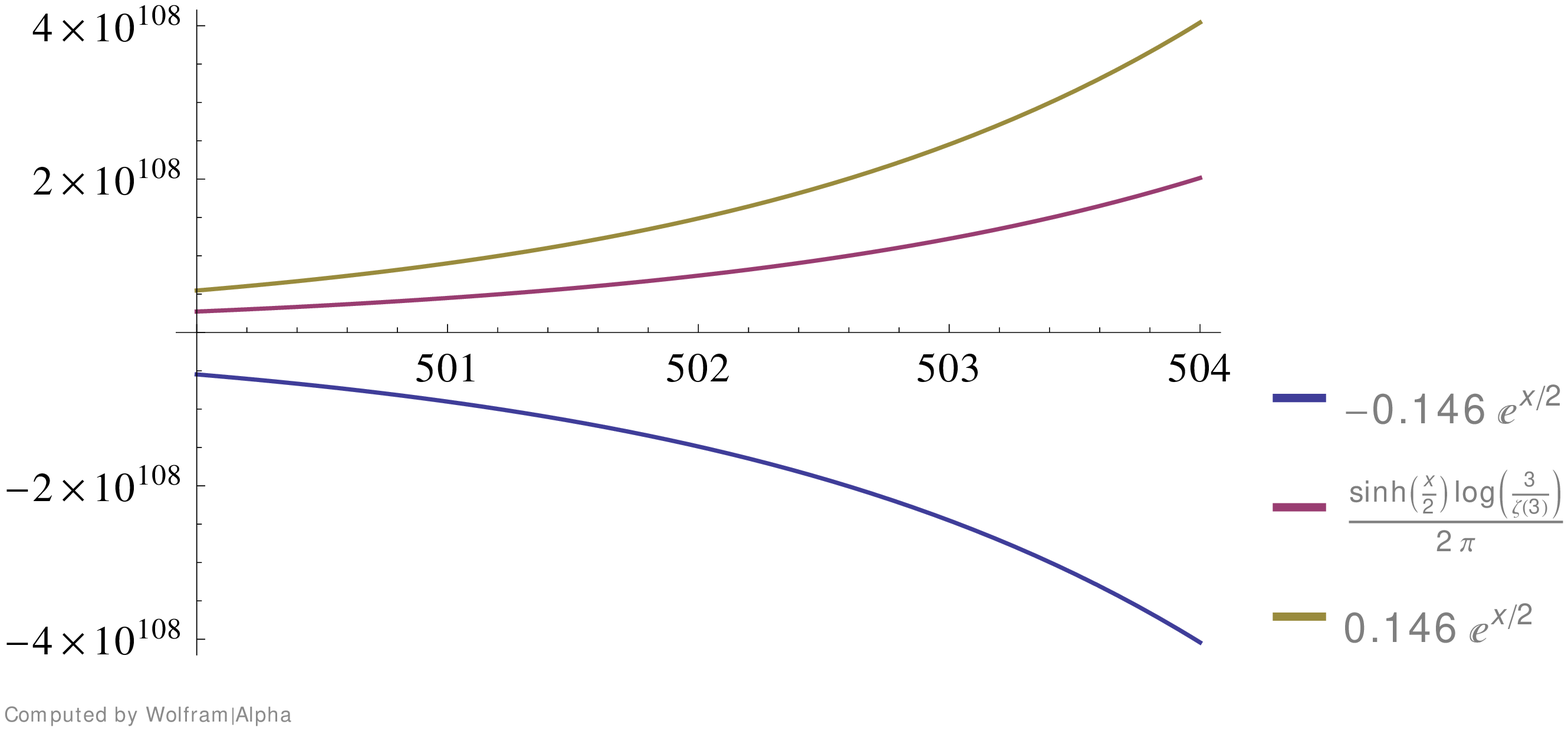}    
    \caption{The two--sided bounds for $\Omega(x)$ (Theorem 5.3. part {\bf b)}) with $x \in [500, 504]$; legend: upper 
		bound -- brown, lower bound -- blue. Here the approcimations $\frac1{2\pi}\,\log\frac3{\zeta(3)} \approx 0.146$ and $\Omega(x) 
		\sim \tfrac1{2\pi}\,\log\tfrac3{\zeta(3)}\,\sinh \left( \tfrac x2\right)$ (violet) are used.}
		\end{center} 
\end{figure}

Obviously, by observing Figure 1 we conclude that the elegant but not so precise bilateral bounding inequality exposed in 
Theorem 5.3 could perhaps be improved. Having in mind connections (see e.g. Theorem 5.2 and its proof) between the Omega function and 
the alternating Mathieu series $\widetilde S_2$, various estimates for $\Omega$ and their efficiency have been considered in 
\cite{PTL}; these approximants are consequences of the \v Caplygin differential inequality. However, we will not consider this 
question here, since it deserves a separate retrospect. 
					
Further, power series characterization of the complete Omega function are established in the sequel.

\begin{theorem}
For the complete Omega--function
   \[ \Omega(z) := 2 \int_{0+}^{\frac12} \sinh (z u) \cot (\pi u)\, {\rm d}u, \qquad \big( z \in \mathbb C\big)\, ,\] 
there hold the properties:
\begin{enumerate}
\item[${\bf (i)}$] {\rm (Partial fraction expansion)}
   \[ \Omega(z) = 2{\rm i} \sinh\Big( \frac z2\Big) \underset{k \in \mathbb Z}{\sum\nolimits_e} \, 
	                \dfrac{(-1)^{k-1} {\rm sgn}(k)}{z+2\pi{\rm i}k}; \]
\item[${\bf (ii)}$] {\rm (Taylor--series expansion ${\rm I.}$)}
   \begin{equation} \label{X81}
	    \Omega(z) = \sum_{k\in \mathbb N_0}\, \Omega_{2k+1} \frac{z^{2k+1}}{(2k+1)!}\,, \qquad (z \in \mathbb C)\,, 
	 \end{equation}
where $\Omega_k$ are the moments of $2\,\cot(\pi u)$, thus
   \[ \Omega_k = 2 \int_{0+}^{\frac12} u^{k} \cot (\pi u)\, {\rm d}u = D_z^k\, \Omega(z)\big|_{z=0}\, ;\]
\item[${\bf (iii)}$] {\rm (Taylor--series expansion ${\rm II.}$)} 
   \begin{equation} \label{X9} 
	    \Omega(z) = \sum_{k \in \mathbb N_0} \frac1{2^{2k}}\, \left\{\sum_{n=0}^k \frac{(-1)^n\,\eta(2n+1)}
			             {\pi^{2n+1}\,(2(k-n)+1)!}\right\}\,z^{2k+1}\,,\qquad (|z|<2\pi);
	 \end{equation}
\item[${\bf (iv)}$] {\rm (Mirror symmetry formula)}
   \[ \Omega(\overline{z}) := \Omega(x+{\rm i}y) = \overline{\Omega(z)}\,, \qquad (z \in \mathbb C);	\]
\item[${\bf (v)}$] {\rm (Reflexivity properties)}
   \begin{align*}
	    \Re\, \Omega(x+{\rm i}y) &= \Re\, \Omega(x-{\rm i}y), \\
	    \Im\, \Omega(x+{\rm i}y) &= - \Im\, \Omega(x-{\rm i}y), \\
			\Im\, \Omega(x) = 0, &\quad \Re\,\Omega(\pm {\rm i}y) = 0.
	 \end{align*}
\end{enumerate}
\end{theorem}

\begin{proof} 
Whereas property {\bf (i)} is a basic property of the paper, {\bf (ii)} follows by expanding $\sinh(uz)$ into its power series,
and substituting it into the definition of $\Omega(z)$. Indeed, 
   \[ \Omega(z) = 2 \int_{0+}^{\frac12} \sum_{k\in \mathbb N_0}\frac{(uz)^{2k+1}}{(2k+1)!}\, \cot (\pi u)\, {\rm d}u 
	              = \sum_{k\in \mathbb N_0}\Omega_{2k+1}\frac{z^{2k+1}}{(2k+1)!}\,, \]
where the legitimate exchange of summation and integration order is applied. Note that $\Omega_k = D_z^k \Omega(z)\big|_{z=0}$. 

As to {\bf (iii)}, firstly let us remark that the Taylor series of the HE--series $\mathfrak h_1(z)$ has been established in 
Proposition 3.2 {\bf a)}, precisely 
   \[ \mathfrak h_1(z) = 2{\rm i} \sum_{n \in \mathbb N_0}(-1)^n\, \eta(2n+1) z^{2n}\,,\qquad (|z|<1) .\]
By {\bf (i)}, it is
   \[ \Omega(z) = -\frac{\rm i}\pi\, \sinh\Big( \frac z2\Big)\, \mathfrak h_1\Big( \frac z{2\pi}\Big), \]
where, in turn, the convergence is assured inside the disk $|z|<2\pi$. Accordingly  
   \begin{align*}
	    \Omega(z) &= \frac2\pi \sum_{m\in \mathbb N_0} \frac{\left(\frac z2\right)^{2m+1}}{(2m+1)!}\,
			             \sum_{n\in \mathbb N_0} (-1)^n \, \eta(2n+1)\,\left(\frac z{2\pi}\right)^{2n}\\
								&= \sum_{m, n\in \mathbb N_0} \frac{(-1)^n\,\eta(2n+1)\, z^{2(m+n)+1}}{2^{2(m+n)}\pi^{2n+1}\,(2m+1)!}\,.
	 \end{align*}
Eliminating $m$ in the double sum, which becomes a simple one with respect to $k=m+n; k\in \mathbb N_0,\, 0 \leq n \leq k$ we get
	 \[	\Omega(z) = \sum_{k \in \mathbb N_0} \frac1{2^{2k}}\, \left\{\sum_{n=0}^k \frac{(-1)^n\,\eta(2n+1)}
			             {\pi^{2n+1}\,(2(k-n)+1)!}\right\}\,z^{2k+1}\, .	\] 
As to {\bf (iv)} and {\bf (v)}, observe that 
   \begin{align*}
	    \Omega(x\pm {\rm i}y) &= 2 \int_{0+}^{\frac12} \sinh u(x\pm {\rm i}y) \cot (\pi u)\, {\rm d}u \\
	                          &= 2 \int_{0+}^{\frac12} \sinh (ux) \cos(uy) \cot (\pi u)\, {\rm d}u  \\ 
                            &\qquad \pm 2{\rm i} \int_{0+}^{\frac12} \cosh (ux) \sin(uy) \cot (\pi u)\, {\rm d}u\, .
	 \end{align*}
In particular there follows
   \begin{align*}
	    \Re\,\Omega(x\pm {\rm i}y) &= 2 \int_{0+}^{\frac12} \sinh u(x\pm {\rm i}y) \cot (\pi u)\, {\rm d}u \\
	    \Im\,\Omega(x\pm {\rm i}y) &= \pm 2\int_{0+}^{\frac12} \cosh (ux) \sin(uy) \cot (\pi u)\, {\rm d}u\, .
	 \end{align*}
Now {\bf (iv)} and {\bf (v)} follow readily. 
\end{proof} 
 
Since inside the disk $|z|<2\pi$ the function $\Omega(z)$ possesses a unique Taylor expansion, both expansions 
\eqref{X81} and \eqref{X9} coincide there. So, equating the coefficients, we deduce the finite closed form expression for 
the moments of $\cot(\pi x)$, namely

\begin{corollary}
The moments $\Omega_k$ can be expressed as 
   \begin{equation} \label{X10}
	    \Omega_{2k+1} =: 2\int_{0+}^{\frac12} u^{2k+1} \cot (\pi u)\, {\rm d}u = \frac{(2k+1)!}{2^{2k}}\, 
	    \sum_{n=0}^k \frac{(-1)^n\,\eta(2n+1)}{\pi^{2n+1}\,(2(k-n)+1)!}\,, 
	 \end{equation}
for all $k \in \mathbb N_0$; thus in particular,
   \[ \Omega_1 = \frac{\eta(1)}{\pi}, \, \Omega_3 = \frac14\Big(\frac{\eta(1)}{\pi} - \frac{6\eta(3)}{\pi^3}\Big), \, 
	    \Omega_5 = \frac1{16} \Big( \frac{\eta(1)}{\pi} - \frac{20\eta(3)}{\pi^3} + \frac{120 \eta(5)}{\pi^5}\Big). \]
\end{corollary}

The authors could not find representations of these important moments as a finite sum in tables of sums and integrals. However, 
there does exist a representation of $\Omega_{2k+1}$ in terms of an infinite series, namely
   \[ \Omega_{2k+1} = \frac1{2^{2k}\pi} \Bigg\{ \frac1{2k+1} + \sum_{n \in \mathbb N} \frac{(-1)^n B_{2n} \pi^{2n}}
	                    {(2n)! (2k+2n+1)!}\Bigg\}, \qquad (k \in \mathbb N_0)\, ,\] 
see \cite[p. 333, Eq. 13 d]{Grobner} and \cite[p. 428, Eq. {\bf 3.748} 2.]{GR}. 

\section{Early ideas of Bernoulli, Euler and Ramanujan; some conjectures}

In order to observe to contribution of the innovative Ramanujan but also of the great Jacob I. Bernoulli, we first need to 
observe the famous representation of the Riemann Zeta function, actually due to Euler (1735/39) \cite{Euler}, with the Bernoulli 
numbers. Indeed, we have shown (see \cite[p. 62 {\em et seq.}]{Butzer})

\begin{theorem} Let $\alpha \in \mathbb C$. Then

\noindent ${\bf a)}$ There holds for all $\alpha \neq 2m+1,\, n \in \mathbb N$
   \[ \zeta(\alpha) = - \sec \Big( \frac{\alpha\pi}2 \Big) \cdot 2^{\alpha-1}\pi^\alpha\, \frac{B_\alpha}{\Gamma(\alpha+1)}; \]
\noindent ${\bf b)}$ There holds for all $\Re(\alpha) > 1$ 
   \[ \zeta(\alpha) = {\rm cosec} \Big( \frac{\alpha\pi}2 \Big) \cdot 2^{\alpha-1}\pi^\alpha\, \frac{B_\alpha^\sim}{\Gamma(\alpha+1)}. \]
\noindent ${\bf c)}$ {\rm Counterpart of Euler's formula}. There holds for odd arguments $2m+1, m\in \mathbb N$ 
   \[ \zeta(2m+1) = (-1)^m 2^{2m} \pi^{2m-1} \, \frac{B_{2m+1}^\sim}{(2m+1)!}.\]
\noindent ${\bf d)}$ {\rm Euler's closed form representation of $\zeta(2m)$}. There holds for even arguments 
   \[ \zeta(2m) = (-1)^{m+1} 2^{2m-1} \pi^{2m} \, \frac{B_{2m}}{(2m)!}.\]
\end{theorem}

Firstly, the numbers $B_n = B_n(0)$ occuring in Euler's representation {\bf d)}, defined in terms of the Bernoulli 
polynomials $B_n(x)$ {\em via} their exponential generating function \eqref{B-1}, where introduced by Jacob I. Bernoulli 
-- prior to 1695 -- published posthumously in 1713 in the second chapter of his {\em Ars Conjectandi} \cite{JIB}. 

In the counterpart {\bf c)} for all arguments, the $B_n$ has been replaced by the conjugate Bernoulli numbers $B_{2m+1}^\sim$, 
defined in terms of the Hilbert transform. 

As to Ramanujan, he introduced the "sign--less" fractional Bernoulli numbers $B_\alpha^*$ in terms of
   \[ B_\alpha^* = \frac{2\Gamma(\alpha+1)}{(2\pi)^\alpha}\, \zeta(\alpha)\,,\]
so that $B_{2m}^* = (-1)^{m+1}\, B_{2m} > 0$ for $\alpha = 2m, m \in \mathbb N$. Thus, he avoided to find substitute for $(-1)^{m+1}$ 
in part {\bf d)} of the previous theorem (see Berndt \cite[p. 125]{Berndt1}).  

There exists a short contribution by J.W.L. Glaisher \cite{JWL} who defined fractional sign--less Bernoulli numbers, similarly as 
Ramanujan did, {\em via} the "Euler" formula
   \[ B_\alpha^* = \frac{2\Gamma(2\alpha+1)}{(2\pi)^\alpha}\, \zeta(2\alpha)\,.\]
In fact, Euler himself (c.f. \cite[p. 351]{Berndt})  had already proceeded in this way for the particular case $\alpha = \tfrac32$ and 
$\alpha = \tfrac52$. He set 
   \begin{align*}
	    p &= \frac3{2\pi^3} \Big\{ 1 + \frac1{2^3} + \frac1{3^3} + \cdots \Big\} \approx 0.05815227, \\
      q &= \frac{15}{2\pi^5} \Big\{ 1 + \frac1{2^5} + \frac1{3^5} + \cdots \Big\} \approx 0.025413275,
	 \end{align*}
and defined
   \[ B_{\frac32} = p, \qquad B_{\frac52} = q\,.\]
The problem of evaluating $\zeta(2m+1)$, at odd integer values, first formulated by P. Mengoli in 1650 (see \cite[p. 125]{Berndt}), 
cannot be solved by replacing the $B_{2m}$ in part {\bf d)} by $B_{2m+1}$ as $B_{2m+1} \equiv B_{2m+1}(0) = 0$ for all 
$m \in \mathbb N$. There exist further articles published more recently, namely by B\"ohmer \cite{Bohmer}, Sinocov \cite{Sinocov}, 
Jonqui\`ere \cite{Jonquiere} and Mus\`es \cite{Muses}.

At Aachen, we discovered the structural closed form solution of Mengoli's question by replacing $B_{2m+1}$ by the {\em conjugate} 
Bernoulli numbers $B_{2m+1}^\sim(0)$, which do {\em not} vanish for all odd integer values. 

The question arises as to what did Ramanujan (and Glaisher) actually mean with $B_\alpha^*$ (and $B_\alpha$)? For this purpose 
at first some words concerning the $B_\alpha(x)$ with fractional indices, not discussed in Section 2. At Aachen the 
{\it Bernoulli functions} $B_\alpha(x)$ with index $\alpha \in \mathbb C, \Re(\alpha)>0$ were defined first for $x \in [0,1)$ by 
$B_\alpha(x) = \mathscr B_\alpha(x)$, where $\mathscr B_\alpha(x)$ is the periodic Bernoulli function, given by
   \begin{equation}  \label{H2}
	    \mathscr B_\alpha(x) = -2\Gamma(\alpha+1) \sum_{k \in \mathbb N} \frac{\cos\pi(2 kx - \frac\alpha2)}{(2\pi k)^\alpha}, 
			                       \qquad (x \in \mathbb R),
	 \end{equation}
with $x \neq 0$ if $\Re(\alpha) \in (0,1]$. In addition, for $\Re(\alpha)>1$, $\mathscr B_\alpha := \mathscr B_\alpha(0)$, is called 
$\alpha$--th {\it Bernoulli number}.

Now, $\mathscr B_\alpha(x), x \in \mathbb R$ is a holomorphic function of $\alpha$ for $\Re(\alpha)>1$, even holomorphic for 
$\Re(\alpha)>0$ when $x \in \mathbb R \setminus \mathbb Z$. The periodic functions $\mathscr B_\alpha(x)$, $x \in [0,1)$ were then 
extended to $\mathbb R$ such that $\mathscr B_\alpha(x)$ interpolates the classical Bernoulli polynomials $B_n(x)$ for all $\alpha = n 
\in \mathbb N$, and were then denoted by $B_\alpha(x)$, for all $x \in \mathbb R$. They were then extended to $B_\alpha(z)$, with  
$\mathbb C \setminus \mathbb R_0^-$, for arbitrary $\alpha \in  \mathbb C$ (see below). The $B_\alpha(z)$ led the way to the 
assertions {\bf a)} and {\bf d)} of Theorem 6.1, the former for $\alpha \neq 2m+1$, the latter for $\alpha = 2m$.  

In order to solve parts {\bf b)} and {\bf c)}, we introduced the Hilbert transform of $\mathscr B_\alpha(x)$, $\mathscr B_\alpha^\sim(x) 
\equiv \mathscr H_1 \big[\mathscr B_\alpha(\cdot)\big](x)$, set up the Fourier series of $\mathscr B_\alpha^\sim(x)$, and proceeded to 
$B_\alpha^\sim(x)$, their periodic extensions to $\mathbb R$. 

The Bernoulli functions, first defined for all $\Re(\alpha)>0$ and $x \in \mathbb R$, were then extended by analytic continuation 
to all $\alpha \in \mathbb C$ and $x=z \in \mathbb C \setminus \mathbb R_0^-$ by the contour integral representation 
   \[ B_\alpha(z) = \frac{\Gamma(\alpha+1)}{2\pi {\rm i}} \int_{\mathfrak C} \frac{u{\rm e}^{uz}}{{\rm e}^u-1}\, 
	                  \frac{{\rm d}u}{u^{\alpha+1}}, \qquad (\Re(z) >0, \, \alpha \in \mathbb C);\] 
here $\mathfrak C$ denotes the positively oriented loop around the negative real axis $\mathbb R^-$, which is composed of a circle 
$C(0; 2r)$ centered at the origin and of radius $2r\, (0 < r < \pi)$ together with the lower and upper edges $C_1$ and $C_2$ of
the complex plane cut along the negative real axis $\mathbb R^-$ (see  Figure 3.). 

\begin{figure}[!ht]
    \begin{center}
        \includegraphics[width=0.64 \textwidth]{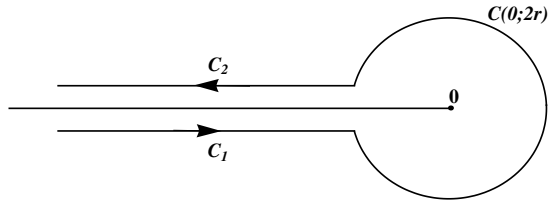}    
    \caption{Integration path $\mathfrak C = C_1 \cup C(0; 2r) \cup C_2$}
		\end{center} 
\end{figure} 

The $B_\alpha(x)$ coincide with the classical Bernoulli polynomials in the case $\alpha = n \in \mathbb N_0$. Indeed, according to 
the Cauchy integral formula for derivatives, noting that $C_1 = -C_2$, one has for $\Re(z)>0$ 
   \[ \frac{\Gamma(\alpha+1)}{2\pi {\rm i}} \int_{\mathfrak C} \frac{u{\rm e}^{uz}}{{\rm e}^u-1}\, 
	                  \frac{{\rm d}u}{u^{\alpha+1}} = \left(\frac{\rm d}{{\rm d}u} \right)^n_{u=0} 
										\left( \frac{u{\rm e}^{uz}}{{\rm e}^u-1}\right) = B_n(z),\]
the last equality following by the defining generating function.

This definition was then extended to $z \in \mathbb C \setminus \mathbb R_0^-$, and it turned out to be consistent with the 
classical $B_n(z)$, with which they coincide for $\alpha = n$. 

It is our conjecture that the Bernoulli functions $B_\alpha^\sim(z)$ can also be defined for $\alpha \in \mathbb C$ and 
$z \in \mathbb C \setminus \mathbb R_0^-$ in terms of the {\it contour integral} 
   \begin{equation} \label{H3}
	    B_\alpha^\sim(z) = - \frac{\Gamma(\alpha+1)}{2\pi {\rm i}} \int_{\mathfrak C} \frac{u{\rm e}^{uz}\,\Omega(u)}{{\rm e}^u-1}\, 
	                       \frac{{\rm d}u}{u^{\alpha+1}}.
	 \end{equation}
It is based upon the fact that, formally,
   \begin{align*} 
	    \mathscr H_1\Big[\frac{u{\rm e}^{u\,\cdot}}{{\rm e}^u-1}\Big](z) &= \frac{u}{{\rm e}^u-1} {\rm PV} 
			           \int_{- \frac12}^{\frac12} {\rm e}^{(z-y)u}\, \cot(\pi y)\, {\rm d} u \\
							&= - \frac{u{\rm e}^{zu}}{{\rm e}^u-1} {\rm PV} 
			           \int_{- \frac12}^{\frac12} {\rm e}^{uy}\, \cot(\pi y)\, {\rm d} y 
							 = - \frac{u{\rm e}^{zu}}{{\rm e}^u-1}\, \Omega(z). 
	 \end{align*}
Let us finally add that given the definition of $B_\alpha^\sim(z)$ {\em via} \eqref{H3} for $\alpha = 2j+1$, and assuming that  
it is correct, one can surely deduce the Fourier expansion of $B_n^\sim(z)$ (found in \cite[p. 32]{Butzer}), using the calculus 
of residues (see e.g. Saalsch\"utz \cite[p. 27]{Saal}, also see \cite[p. 331]{Remmert1}). In that case, the validity of part {\bf b)} 
of Theorem 6.1 could be extended from $\Re(\alpha)>1$ to all $\alpha \in \mathbb C$ (just as for part {\bf a)}). 

Hopefully the values of \eqref{H3} at $z=0$, that is $B_\alpha^\sim(0)$, could then give us more information concerning the 
possible irrationalities of $\zeta(5), \zeta(7), \zeta(9), \cdots$.  

Substituting the Taylor series representation \eqref{X81} of $\Omega(u)$ into $B_\alpha^\sim(z)$ of \eqref{H3} for $\alpha = 2j+1, 
j \in \mathbb N_0$, and interchanging the integral and sum, there results the representation
   \begin{equation} \label{X11}
	    B_{2j+1}^\sim(z) = \frac{\Gamma(2j+2)}{2\pi {\rm i}} \sum_{k \in \mathbb N_0} \frac{\Omega_{2k+1}}{(2k+1)!}\,
	                       \int_{\mathfrak C} \frac{{\rm e}^{uz}}{1-{\rm e}^u} \frac{u^{2k}}{u^{2j}}\, {\rm d}u, 
	 \end{equation}
the latter integral being easier to evaluate than $B_{2j+1}^\sim(z)$ in terms of the original \eqref{H3}. In turn, as inside 
$|u|<2\pi$ the Bernoulli polynomials' generating function 
   \[ \frac{u\,{\rm e}^{uz}}{{\rm e}^u-1} = \sum_{n \in \mathbb N_0} B_n(z)\, \frac{u^n}{n!}\,,\]
by the Cauchy's integral formula we have
   \begin{align} \label{X16}
	    \mathscr J_{j, k}(z) &= \int_{\mathfrak C} \frac{u^{2k+1}\, {\rm e}^{uz}}{{\rm e}^u-1} \frac{{\rm d}u}{u^{2j+1}}  
			                      = \frac{2\pi {\rm i}}{\Gamma(2j+1)} \left(\frac{\rm d}{{\rm d}u} \right)^{2j}_{u=0} 
										          \left( u^{2k}\cdot \frac{u\,{\rm e}^{uz}}{{\rm e}^u-1}\right) \\
													 &= \frac{2\pi {\rm i}}{\Gamma(2j+1)} \left(\frac{\rm d}{{\rm d}u} \right)^{2j}_{u=0} 
										          \left( \sum_{n \in \mathbb N_0} B_n(z)\, \cdot \frac{u^{2k+n}}{n!}\right)\nonumber \\
													 &= \frac{2\pi {\rm i}}{(2j)!} \sum_{n \in \mathbb N_0} \frac{B_n(z)}{n!}\,
													    \frac{(2k+n)!}{(2k+n-2j)!} \,\big(u^{2k+n-2j}\big)_{u=0}\nonumber \\
													 &= \frac{2\pi {\rm i}}{(2j-2k)!}\, B_{2j-2k}(z)\, . \nonumber 
	 \end{align}
So, employing Corollary 5.5, we conjecture that $B_{2j+1}^\sim(z)$ can be represented as the 
following  double finite sum, involving the $B_{2j-2k}$ and the $\eta(2n+1)$, as 
   \begin{align} \label{X17}
	    B_{2j+1}^\sim(z) &= - \frac{\Gamma(2j+2)}{2\pi {\rm i}} \sum_{k \in \mathbb N_0} \frac{\Omega_{2k+1}}{(2k+1)!}\,
			                      \mathscr I_{j,k}(z) \nonumber\\
		                   &= - \frac{(2j+1)!}\pi \sum_{k=0}^j \frac{B_{2j-2k}(z)}{4^k\,(2j-2k)!} \,
											      \sum_{n=0}^k \frac{(-1)^n\,\eta(2n+1)}{\pi^{2n}\,(2(k-n)+1)!}\,. 
	\end{align}

\section*{A portait of Godfrey Harold Hardy (1877--1947)} 

Born February 7, 1877 in Cranleigh,  Surrey, Hardy graduated from Trinity College, Cambridge in 1899, became a fellow at 
Trinity in 1900, and lectured in mathematics there from 1906 to 1919, since 1914 as  Cayley Lecturer. In 1919 he was appointed to 
the Savillian Chair of Geometry at the University of Oxford, spent 1928--29 as visiting professor at Princeton, returned back to 
Oxford, and finally became Sadleirian Professor of Pure Mathematics in Cambridge in 1931. There he remained until his death on 
December 1, 1947, after having retired in 1942. 

Among Hardy's early works  are his popular eleven books, among them {\it Integration of Functions of a Single Variable}, 
CUP (1905) \footnote{CUP - abbreviation of Cambridge University Press, Cambridge, UK, while OUP is the abbreviation for Oxford 
University Press, Oxford, UK}; {\it A Course of Pure Mathema\-tics}, CUP (1908); (10th Edit. 2008, with T. Korner); followed by 
{\it The General Theory of Dirichlet's Series}, CUP (1915), with M. Riesz; {\it Inequalities}, CUP (1934), with J.E. 
Littlewood and G.P\'olya (reprinted in 1952); {\it An Introduction to the Theory of Numbers}, OUP (1938), with E.M. Wright 
(6th Edit. in 2008 with D.R. Heath--Brown and J.H. Silverman); {\it A Mathematician's  Apology}, CUP (1940, 2004); {\it Fourier 
Series}, CUP (1944), with W.W. Rogosinski; {\it Divergent Series}, OUP (1949). Hardy was the author or coauthor of 
more than 300 papers and the recipient of numerous honours. His doctoral students included L. Bosanquet, M. Cartwright, U. Haslam-Jones, 
A.C. Offord,  R. Rado, R. Rankin, K.A. Rau, D. Spencer and E. Titchmarsh. Hardys collaboration with J.E. Littlewood, which set in 
1911 and extended over 30 years, brought fresh impetus into his work. Their collaboration is among the most famous as such in 
mathematics history. 

The sequence of essays by his former students in [{\it J. London Math. Soc.} {\bf 25} (1950)  81--101] is an excellent source on 
Hardy's place in mathematics. See also G.H. Hardy, {\it Collected Papers}, 7 vols. (Clarendon Press,  Oxford, 1966-1979).

According to MacTutor's article on Hardy: "... Hilbert was so concerned that Hardy was not being  properly treated (while living at 
Cambrige) that he wrote to the Master of the College pointing out that {\it the best mathematician in England} should have the 
best rooms." 

Harald Bohr, who stood in close contact with Hardy, assessing the leadership of Hardy and Littlewood in English research (1947), 
wrote "I may report what an excellent colleague once jokingly said: "Nowadays, there are only three really great English 
mathematicians: Hardy, Littlewood, and Hardy--Littlewood." "

The following quotations give an interesting view of Hardy's  thoughts: "I am obliged to interpolate some remarks on a very 
difficult subject: PROOF and its importance in mathematics. All physicists, and a good many quite reputable mathe\-maticians, 
are contemptuous about proof. I heard Professor Eddington, for example, mention that proof, as pure mathematicians understand it, 
is really quite uninteresting and unimportant, and that no one who is really certain that he has found something good should 
waste his time looking for proof." While reviewing the question of the reality of nature from a mathematical viewpoint, 
G.H. Hardy stated: "... I will state my own 
position dogmatically in order to avoid minor misapprehensions. I believe that mathematical reality lies outside us, that our 
function is to discover or observe it, and that theorems which we prove, and which we describe grandiloquently as our "creations" 
are simply our notes of our observations. This view has been held in some form or another, from Plato onwards..."

According to the appraisal of the Trinity College Chapel, Hardy "was universally recognized as pre-eminent among the world's best 
mathematicians"; "As his most important influence Hardy cited the self-study of {\it Cours d'Analyse de l'Ecole Polytechnique} by the 
French mathematician Camille Jordan ... Hardy  is credited with reforming British mathematics by bringing rigour into it, which 
was previously a  characteristic of French, Swiss and German mathematics. British mathematics had remained largely in the 
tradition of applied mathematics, in thrall to the reputation of Isaac Newton (see Cambridge  Mathematical Tripos). Hardy was more 
in tune with the {\it cours d'analyse} methods dominant in France, and aggressively promoted his conception of pure mathematics, 
in particular  against the  hydrodynamics which was an important part of Cambridge  mathematics." 

More significant than Hardy's collaboration with Littlewood was that with Srinivasa Aiyangar Ramanujan, born 1887 at Erode, 
Madras Presidency \footnote{Today Tamil Nadu.} (of a Tamil Brahmin family), a self--taught and obsessive shipping clerk from Madras (see 
\cite{Alladi}). In 1913 he sent a nine--page paper to Hardy, dealing with two remarkable, novel infinite series of hypergeometric type  
(related to research of Euler and Gauss), and continued fractions. Hardy was so amazed that he commented to Littlewood that Ramanujan 
was "a mathematician of highest quality, a man of altogether exceptional originality and power". Hardy brought him to Cambridge 
(the well--established H.F. Baker and E.W. Hobson had returned the papers without comment), made him aware of modern mathematics 
and so provided  a solid  foundation to Ramanujan's inventiveness. They became friends, collaborated (called  "the one romantic 
incident in my life" by Hardy) and wrote five remarkable papers together. 

Ramanujan made extraordinary contributions to mathematical analysis, number theory, continued fractions and infinite series, 
rediscovered known theorems of Bernoulli, Euler, Gauss and Riemann. He conjectured or proved nearly 3900 theorems, identities and 
equations. Hardy regarded the "discovery of Ramanujan" as "his greatest contribution to mathematics",  and assailed Ramanujan's 
natural geniuity, being on the same league as Euler and Gauss. In his book {\it Ramanujan's twelve lectures on subjects suggested 
by his life  and work} (Cambridge, 1940), Hardy placed an everlasting monument for Ramanujan who had died too young. Since 
Bernoulli polynomials and numbers play an essential role in the present paper, and Ramanujan wrote his 
first formal paper (at age 17, but published it later) [Some properties of Bernoulli's numbers, {\it J. Indian Math. Soc.} 
{\bf 3} (1911) 219--234] this paper is dedicated to Hardy, his discoverer. For an overview of Bruce Berndt's excellent work 
in matters Ramanujan, see his {\it An overview of Ramanujan's notebooks} \cite{Berndt1} and {\it Ramanujan, his lost notebooks, 
its importance} \cite{Berndt2}; also see \cite{Alladi}.

For his lecture tour of ten universities in UK, which was kindly organized by Lionel Cooper (1915--1979), P.L.B. offered 
five possible topics, one being "On some theorems of Hardy, Littlewood and Titchmarsh". Surprisingly, no university in 
the England or Scotland picked it. At a special dinner at Chelsea College at the beginning of the tour, in the presence of 
about a dozen of London's major analysts, one remarked that a reason could be that Hardy is/was no longer regarded as 
Britain's top mathematician at the time. Even students (whom P.L.B. also regarded as great British analysts) were assigned to 
the same category. Peculiarly, no one present reacted to the astonishing assertion or countered it.

Both authors are grateful to Maurice Dodson for inviting them to his conference Fourier Analysis and Applications, held 
at York in 1993, where both experienced three of Britain's best analysts at the time, namely James (Jim) Gourlay Clunie 
(1926--2013) \cite[pp. 108--109]{BSS}, Walter Hayman and Frank Bonsall (1920--2011). The former two, spent their 
"retirement" at York University, the latter at the spa town of Harrogate. 
\begin{center}
* \qquad \qquad * \qquad \qquad *
\end{center}
In memory of Pater Wilhelm Brabender, OMI (1879--1945), a paternal great-uncle of P.L.B., an Oblate missionary 
in Saakatchevan and  British Columbia (Canada) from 1905 to 1931, who preached in French, German and English, was an authority on 
Cree  ethnography, and Rev. Pater Julius Pogany, S.J. (1909-1986), paternal uncle of T.K.P., who was a Jesuit 
missionary and teacher at St. Aloysius' College in Galle (Sri Lanka) from 1949 to 1963 (teaching in both English and Sinhalese), 
had received his MSc in Theoretical Mathematics and Physics at Royal Yugoslav University in Zagreb in 1936, and had completed his Jesuit education at 
the Pontifical Gregorian University, Rome, in 1944. 

\section*{Acknowledgments} 

The authors express their thanks to the referee for his careful study of the manuscript. Also they express deep gratitude to 
Djurdje Cvijovi\'c (Vin\v ca Nuclear Research Institute, Belgrade) for his 
valuable time spent in trying to solve the author's (probably unsolvable) conjecture on expressing the $\coth$ function as a 
linear combination of digammas, as well as to Gradimir V. Milovanovi\'c (Mathematical Institute SASA, Belgrade) for preparing 
Figures 1 and 2. The authors also wish to thank Rudolf L. Stens  (Lehrstuhl A f\"ur Mathematik, RWTH Aachen) for scanning very many 
hand--written pages of P.L.B., in the course a year or so, and emailing them all to T.K.P. 
\allowdisplaybreaks

\end{document}